\documentclass[11pt,a4paper,twoside]{article}
\usepackage[utf8]{inputenc}
\usepackage[english]{babel}

\usepackage{comment}

\usepackage[top=1in, bottom=1in, left=1.25in, right=1.25in]{geometry}

\usepackage{amssymb}
\usepackage{amsmath}
\usepackage{amsthm}

\usepackage{enumitem}

\usepackage{graphicx}

\newtheorem{theorem}{Theorem}[section]
\newtheorem*{theorem*}{Theorem}
\newtheorem{proposition}[theorem]{Proposition}
\newtheorem{lemma}[theorem]{Lemma}
\newtheorem{corollary}[theorem]{Corollary}

\theoremstyle{definition}

\theoremstyle{example}

\DeclareMathOperator{\reals}{\mathbb{R}}
\DeclareMathOperator{\naturals}{\mathbb{N}}

\newcommand{\norm}[1]{\left\lVert#1\right\rVert}

\title{Rigidity of the unstable foliation}

\author{Sergi Burniol Clotet}

\begin{document}
	
	\maketitle
\begin{abstract}
%	Anosov flows are one of the central examples of hyperbolic and chaotic dynamical systems. Associated to them are the stable and unstable foliations, whose leaves contract and expand exponentially along the flow. 
	We establish a rigidity result for the unstable foliations of transitive Anosov flows on 3-manifolds: if the unstable foliations of two such flows are equivalent (that is, if there exists a homeomorphism mapping one foliation to the other), then the flows are topologically conjugate up to a constant change of time. This result partially generalizes earlier rigidity theorems for horocyclic flows on compact surfaces of negative curvature, originating in the work of Ratner. In that setting, it is known that equivalence of unstable foliations implies that the underlying surfaces are homothetic.
	
%	In this setting, there are previous rigidity results due to Ratner, Marcus, and Abe, among others. The latter showed that if the unstable foliations of two such flows are equivalent, then the underlying surfaces are homothetic.
\end{abstract}
%	
%Reference: Thierry Barbot
%Ver articulo de Abe sobre factors...
%Feldman- Ornstein do the contact case

%escribo la pregunta aca??
\section{Introduction}

Understanding whether two dynamical systems are equivalent is a classical question in the field. Several authors have addressed this question in the case of horocyclic flows. Ratner established a rigidity result for horocyclic flows of finite-volume hyperbolic surfaces in a measurable setting \cite{Ratner82}. Later, Marcus obtained a topological rigidity result for horocyclic flows of compact surfaces with variable negative curvature \cite{Marcus83}. Abe proved a stronger statement using different methods and Otal's marked length spectrum rigidity theorem \cite{Abe95, Otal90}. Other works related to this matter are \cite{Ratner1982Factors,Ratner1983,FeldmanOrnstein1987,Abe1991RigidityFactors}.

Two foliations are called equivalent if there exists an homeomorphism that maps leaves of one foliation to leaves of the other. A precise formulation of topological rigidity for horocyclic flows is the following:
\begin{theorem}\cite{Marcus83,Abe95}\label{thmAbe}
	Let $(S,g),(S',g')$ be two negatively curved compact surfaces. If the unstable foliations on $T^1S$ and $T^1 S'$ are equivalent, then the two surfaces $(S,g)$ and $(S',g')$ are homothetic.
\end{theorem}

In this article, we establish a rigidity result in the setting of transitive Anosov flows:
\begin{theorem}
	Let $f_t, g_t$ be two $C^1$ transitive Anosov flows on a compact $3$-manifold $M$. Assume that the strong unstable foliations $\mathcal{W}^u_f$ of $f_t$ and $\mathcal{W}^u_g$ of $g_t$ are equivalent. Then there exists $\lambda>0$ such that the flows $f_t$ and $g_{\lambda t}$ are topologically conjugate.
\end{theorem}

We observe that topological conjugacy is not persistent under perturbation in the space of pairs of Anosov flows. In particular, given an Anosov flow, a nonconstant change of time will produce nonequivalent unstable foliations. Hence the moduli space of nonequivalent unstable foliations is infinite-dimensional. The behavior of the unstable foliation contrasts with the behavior of the Anosov flow, since two $C^1$-close Anosov are orbit equivalent. Barbot showed that two transitive Anosov flows are orbit equivalent if and only if their center unstable foliations are equivalent \cite{Barbot95}. A statement in this direction had already been announced by Brunella in \cite{Brunella92}.

A natural question is whether the regularity of the conjugacy can be improved, leading to a stronger form of rigidity. Recent work by Gogolev, Leguil, and Rodriguez-Hertz shows that, under suitable hypotheses, a topological conjugacy between Anosov flows can in fact be upgraded to higher regularity \cite{GogolevRH,GogolevLeguilRH}.

%comment about topological conj + many nonequivalent 

In the proof of Theorem \ref{thmAbe}, one first establishes the topological conjugacy between geodesic flows and then invokes marked length spectrum rigidity. Existing proofs of the topological conjugacy rely on properties that are highly specific to geodesic flows and extend at most to contact Anosov flows. The novelty in this paper is an argument that applies to any transitive Anosov flow on a $3$-manifold.

The proof is divided into two cases. According to Plante's alternative an Anosov flow is either a constant time suspension or topologically mixing. The first case is dealt with a topological argument. The second one requires deeper dynamical arguments and relies on the Anosov structure, namely it uses the non-integrability of the stable and the unstable foliations.

%comment problems in abe marcus
%contact case feldman ornstein
%comment higher dimension
%comment gogolev leguil RH ( apparently on arxiv)
%comment integrability problem
%organization?
\section*{Acknowledgement}

The question of rigidity of the unstable foliation arised in conversations with Matilde Martínez and Rafael Potrie. The author is greatly thankful for the discussions and comments on the paper.

\section{Background}

Let $M$ be a smooth compact connected Riemannian manifold. A $C^1$ non-singular flow $f_t:M\to M, \, t\in \reals$, is Anosov if there exists a $Df_t$-invariant continuous splitting $TM = E^s \oplus \reals X \oplus E^u$, where $X = \frac{d f_t}{dt}|_{t=0}$ is the vector field associated to $f_t$, and the bundles $E^s$ and $E^u$ are uniformly contracted and expanded, respectively, meaning that there exist constants $C>0$ and $0<\lambda <1$ such that:
\begin{enumerate}
	\item for every $v\in E^s$ and $t\ge0$, $\norm{Df_t(v)} \le C\lambda^ t \norm{v}$,
	\item for every $v\in E^u$ and $t\ge0$, $\norm{Df_{-t}(v)} \le C\lambda^ {t} \norm{v}$.
\end{enumerate}

The Stable and Unstable Manifold Theorem provides submanifolds $W^s(x)$ and $W^u(x)$ for each point $x\in M$ which are tangent to $E^s$ and $E^u$, respectively, at every point. If $d$ is the Riemannian distance on $M$, the stable and the unstable manifold of $x$ can be charecterized by 
\begin{align*}
	W^s(x)&=\{y\in M \mid d(f_t(x),f_t(y)) \to 0 , \, t\to +\infty\},
	\\
	W^u(x)&=\{y\in M \mid d(f_t(x),f_t(y)) \to 0 , \, t\to -\infty\}.
\end{align*}

The center stable and unstable manifolds are defined as
\begin{align*}
	W^{cs}(x)&=\cup_{t\in \reals} f_t(W^s(x)),
	\\
	W^{cu}(x)&=\cup_{t\in \reals} f_t(W^u(x)).
\end{align*}
These four kinds of submanifolds form four continuous $f_t$-invariant foliations of $M$ that we denote by $\mathcal{W}^s_f, \, \mathcal{W}^u_f,\, \mathcal{W}^{cs}_f,\, \mathcal{W}^{cu}_f$ accordingly.

We denote by $W_{\varepsilon}^s(x)$ the ball in $W^s(x)$ of center $x\in M$ and radius $\varepsilon>0$ in the metric on $W^s(x)$ induced by the Riemannian distance $d$ on $M$. We use the same notation for the corresponding local balls in the other invariant foliations.

\begin{theorem}[Local product structure]  \label{local_prod}
	Let $f_t : M \to M$ be an Anosov flow. There exist $\varepsilon>0$ and $\delta>0$ such that for every $x,y\in M$ with $d(x,y)<\delta$, there is a unique point  
	\[
	[x,y] \;\in\; W^s_{\varepsilon}(x)\;\cap\; W^{cu}_{\varepsilon}(y).
	\]  
	Moreover, the map $(x,y)\mapsto [x,y]$ is continuous on its domain.
\end{theorem}
The roles of the stable and the unstable can be reversed in the previous theorem.

We next recall Plante's alternative on transitive Anosov flows.

\begin{theorem}\cite[Theorem 1.8]{Plante72} \label{alternative}
	Let $f_t:M\to M$ be a transitive Anosov flow. Then one and only one of the following is true:
	\begin{enumerate}
		\item $f_t$ is topologically mixing and the foliations $\mathcal{W}^s$ and $\mathcal{W}^u$ are minimal.
		\item $f_t$ is the constant time suspension of an Anosov diffeomorphism on a compact $C^1$ submanifold of codimension one in $M$.
	\end{enumerate}
\end{theorem}

Let $x,y\in M$ such that $y\in W^s_{\delta}(x)$. There exists $0<\varepsilon' \le \varepsilon $ such that, for all $x'\in W^{cu}_{\varepsilon '} (x)$, we have $d(x', y)<\delta$, so $[x',y]$ is well defined. The stable holonomy is defined as
\begin{equation*}
	\begin{matrix}
		H^s_{x,y} :& W^{cu}_{\varepsilon'} (x)& \longrightarrow & W^{cu}_{\varepsilon} (y)  \\
		& x' & \longmapsto & [x',y].
	\end{matrix}
\end{equation*}
We say that $\mathcal{W}^s$ and $\mathcal{W}^u$ are jointly integrable if, whenever $d(x,y)<\delta$, 
\[
H_{x,y}^s ( W^u _{\varepsilon'}(x)) \subset W^u _{\varepsilon}(y).
\]
It is clear that, if $f_t$ is a constant time suspension of an Anosov diffeomorphism, then $\mathcal{W}^s$ and $\mathcal{W}^u$ are jointly integrable. In fact, Plante also shows:

\begin{theorem}\cite[Theorem 3.7]{Plante72} \label{ji}
	Let $f_t:M\to M$ be an Anosov flow such that $\mathcal{W}^s$ and $\mathcal{W}^u$ are jointly integrable and either $\dim E^s=1 $ or $\dim E^u =1$. Then $f_t$ is a constant time suspension of an Anosov diffeomorphism on a compact $C^1$ submanifold of codimension one in $M$.
\end{theorem}
The Anosov diffeomoprhism appearing on the previous theorem is topologically conjugate to a hyperbolic toral automorphism \cite{Franks70}.

We now restrict to the case that $M$ is a $3$-manifold. The foliations $\mathcal{W}^s$ and $\mathcal{W}^u$ have both dimension $1$, so they can be parametrized by flows on $M$. We will choose a particular parametrization introduced by Marcus which will make later computations easier.

%como tratar el caso no orientable?

\begin{theorem} \label{parametrization}Let $f_t:M\to M$ be a transitive Anosov flow on a compact $3$-manifold $M$. Assume that $E^s$ and $E^u$ are orientable. Then there exist a constant $h>0$ and continuous flows $h^+_t:M \to M$ and $h^-_t : M\to M$ such that 
	\begin{enumerate}
		\item for every $x\in M$, the orbit $h^+_{\reals}(x)$ is the stable manifold $W^s(x)$ and the orbit $h^-_{\reals}(x)$ is the unstable manifold $W^u(x)$,
		\item for every $t,s,u\in \reals $, we have
		\[
		h^+_s \circ f_t = f_t \circ h^+_{se^{ht}} , \quad h^-_u \circ f_t = f_t \circ h^-_{u e^{-ht}} .
		\]
	\end{enumerate}
	
\end{theorem}

\begin{proof}
	Under the assumption that $f_t$ is topologically mixing, Margulis constructed two families of measures on the stable and the unstable manifolds which are uniformly contracted and expanded by $f_t$, as part of his construction of the measure of maximal entropy of $f_t$ \cite{Margulis70}. Later Marcus used these families of measures to define the continuous flows $h_t^+$ and $h_t^-$ with the proprerties above \cite{Marcus75}.
	
	The other possibility is that $f_t$ is a constant time suspension of an Anosov diffeomorphism $f$. Sinai constructed two families of measures on the stable and the unstable manifolds of $f$ which are uniformly contracted and expanded by $f$ \cite{Sinai68,RuelleSullivan}. These measures can be lifted to a family of measures on the stable and unstable manifolds of the suspension $f_t$ with the exact same properties as the Margulis measures. Parametrizing each foliation by the corresponding family of measures yields the two flows $h^+_t$ and $h^-_t$ with the properties above.
\end{proof}

\section{Proof of the rigidity}

\subsection{Topologically mixing Anosov flows}

%orientable hyp missing

Let $f_t, g_t$ be two $C^1$ transitive Anosov flows on a compact $3$-manifold $M$. Assume that there is a homeomorphism $\phi: M\to M$ which sends the unstable foliation $\mathcal{W}^u_f$ of $f_t$ to the unstable foliation $\mathcal{W}^u_g$ of $g_t$. In other words, for every $x\in M$, we have 
\[
\phi (W^u_f (x)) = W^u_g(\phi(x)).
\]

%\textbf{Situation.} Let $(M,f_t)$ and $(N,g_t)$ be transitive Anosov flows on compact $3$-manifolds. There exist unstable and stable horocyclic flows $h^-$ and $h^+$ with uniform parametrization, i.e.
%\[
%h^-_s \circ g_t = g_t \circ h^-_{se^{-ht}},
%\]
%where $h$ is the topological entropy of $g_t$.  
%Assume there is an orbit equivalence $\phi : M\to N$ between the unstable horocyclic flows.  

\begin{lemma}\label{local_times}
	There exist $\delta > 0$, and continuous functions $t_1, t_2, t_3 : M \times (-\delta,\delta) \to \mathbb{R}$ such that for every $x \in M$ and $t \in (-\delta, \delta)$,
	\begin{equation}\label{basic_formula}
		\phi(f_t(x)) = h^-_{t_3(x,t)} \, g_{t_2(x,t)} \, h^{+}_{t_1(x,t)} \, \phi(x),
	\end{equation}
	where $h_t^+$ and $h_t^-$ are the parametrizations of the stable and the unstable foliation of $g_t$, respectively, defined in Theorem \ref{parametrization}.
\end{lemma}  

\begin{proof}
	Let $\delta_2>0$ such that $g_t$ has local product structure for points at distance less than $\delta_2$ (Theorem \ref{local_prod}). Since $ \phi $ is uniformly continuous there exists $\delta_1>0$ such that $d(x,y)<\delta_1$ implies $d(\phi(x), \phi(y))<\delta_2$, $x,y\in M$. Finally, by uniform convergence, there exists $\delta>0$ such that, if $|t|<\delta$, then $d(f_t(x),x)<\delta_1$. 
	
	This choice of $\delta>0$ guarantees that, for every $(x,t)\in M\times (-\delta,\delta)$, the local stable manifold of $\phi(x)$ intersects the local center unstable manifold of $\phi(f_t(x))$ at the point $[\phi(x), \phi (f_t(x))]$. By Theorem \ref{parametrization}, there exists a unique time $t_1(x,t)\in \reals$ such that 
	\[ [\phi(x), \phi (f_t(x))]= h^{+}_{t_1(x,t)} \, \phi(x). \]
	
	Since $[\phi(x), \phi (f_t(x))]$ is in the center unstable manifold of $\phi(f_t(x))$, there exists a small time $t_2(x,t)$ such that $g_{t_2(x,t)} [\phi(x), \phi (f_t(x))]$ belongs to the unstable manifold of $\phi(f_t(x))$ and also a unique time $t_3(x,t)$ such that
	\[
	h^-_{t_3(x,t)} g_{t_2(x,t)} [\phi(x), \phi (f_t(x))] = \phi(f_t(x)).
	\]
	The continuity of the functions $t_1,t_2$ and $t_3$ follows from the continuity of the local product structure of $g_t$  and the continuity of the flows $f_t, g_t, h_t^+$ and $h_t^-$.
\end{proof}

%Before stating the key result, we observe that we can take local stable and unstable manifolds can be written in terms of the flows $g_t, h^+_t$ and $h^-_t$. More precisely we will consider local center stable manifolds of the form
%\[
%W^{cs}_{loc} (x) =\{ g_t h^+_s x \, | \, |t|<\varepsilon_1,\, |s|<\varepsilon_2 \},
%\]
%with $\varepsilon_1$ and $\varepsilon_2$ small enough.·

Next we state the key result.

\begin{proposition}\label{key}
%Let $f_t, g_t$ be two $C^1$ Anosov flows on a compact $3$-manifold $M$. Assume that both $E^s_g$ and $E^u_g$ are orientable. Let $\phi:M\to M$ be an orbit equivalence between the unstable foliation $\mathcal{W}^u_f$ of $f_t$ and the unstable foliation $\mathcal{W}^u_g$ of $g_t$. Assume that $g_t$ is transitive. Then,
%	\begin{enumerate}
%		\item either $\phi$ sends the center unstable foliation $\mathcal{W}^{cu}_f$ of $f_t$ to the center unstable foliation $W^{cu}_g$ of $g_t$,
%		\item or $\mathcal{W}^s_g$ and $\mathcal{W}^u_g$ are jointly integrable.
%	\end{enumerate} 
%%	If $\phi$ does not send $W^{cu}$ to $W^{cu}$, then $W^s$ and $W^u$ are jointly integrable.
Let $f_t, g_t$ be two $C^1$ Anosov flows on a compact $3$-manifold $M$. Assume that both $E^s_g$ and $E^u_g$ are orientable and assume that $g_t$ is topologically mixing. Let $\phi:M\to M$ be an equivalence between the unstable foliation $\mathcal{W}^u_f$ of $f_t$ and the unstable foliation $\mathcal{W}^u_g$ of $g_t$. Then $\phi$ sends the center unstable foliation $\mathcal{W}^{cu}_f$ to the center unstable foliation $\mathcal{W}^{cu}_g$.
\end{proposition}

\begin{proof}
	The proof is by contradiction. We assume that $\phi$ does not send the center unstable foliation $\mathcal{W}^{cu}_f$  to the center unstable foliation $W^{cu}_g$ and show that then $\mathcal{W}^s_g$ and $\mathcal{W}^u_g$ are jointly integrable. By Plante's Theorems \ref{alternative} and \ref{ji}, this contradicts the fact that $g_t$ is topologically mixing. We observe that $f_t$ is also topologically mixing, since the unstable foliations of $g_t$ and $f_t$ are both minimal.
	
	Let $\delta_2>0$ such that $g_t$ has local product structure for points at distance less than $\delta_2$ (Theorem \ref{local_prod}). We say that $\mathcal{W}^s_g$ and $\mathcal{W}^u_g$ are jointly integrable at a point $x\in M$ if there exists a small constant $\delta_x>0$ such that 
	\begin{itemize}
		\item for every $z\in W^{cu}_{\delta_x} (x)$ and every $y\in W^{s}_{\delta_x} (x)$, we have $d(z,y)<\delta_2$ so $[z,y]$ is well defined,
		\item and, for every $y\in W^{s}_{\delta_x} (x)$, we have $H^s_{x,y} (W^u_{\delta_x}(x)) \subset W^u_{\varepsilon}(y) $.
	\end{itemize} 
	Equivalently, $\mathcal{W}^s_g$ and $\mathcal{W}^u_g$ are jointly integrable at $x\in M$ if there exists a small continuous open surface containing $x$ foliated by stable manifolds and by unstable manifolds, but we do not control the size of the surface.
	With this definition $\mathcal{W}^s_g$ and $\mathcal{W}^u_g$ are jointly integrable if and only if they are jointly integrable at every point $x\in M$. Our goal is to show that they are jointly integrable at every point.
	
	\textit{\textbf{Step 1.} Let $t\in (-\delta,\delta) $ and $x\in M$ such that $\phi(f_t(x))\in W^{cu}_{g,\varepsilon}(\phi(x))$. Then, for all $y\in M$,  $\phi(f_t(y))\in W^{cu}_{g,\varepsilon}(\phi(y))$.
	}
	
	Let $y\in M$. Since the unstable foliation of $f_t$, is minimal there exists a sequence $ x_n$ in $W^u(x)$ converging to $y$. Then $\lim \phi(x_n)= \phi (y) $ and $\lim \phi(f_t(x_n))= \phi (f_t(y)) $. For all $n\in \naturals$, $\phi(f_t(x_n))\in W^{cu}_{g,\varepsilon}(\phi(x_n))$. Taking the limit we obtain $\phi(f_t(y))\in W^{cu}_{g,\varepsilon}(\phi(y))$.
	
	\textit{\textbf{Step 2.} There exists $0<\delta_0<\delta$ such that, for every $t\in (-\delta_0,\delta_0)\setminus\{0\}$ and every $x\in M$, we have $\phi(f_t(x))\not \in W^{cu}_{g,\varepsilon}(\phi(x))$. Moreover $t_1(x,t)$ has different sign for $t\in (-\delta_0,0)$ and $t\in (0,\delta_0)$.
}
	
	We show that if the previous property is not satisfied then $\phi$ sends the center unstable foliation $\mathcal{W}^{cu}_f$  to the center unstable foliation $\mathcal{W}^{cu}_g$. 
	
	The negation of the statement yields a sequence of nonzero times $t_n$ converging to $0$ and a sequence of points $x_n$ in $M$ such that $\phi(f_{t_n} (x_n))\in W^{cu}_{g,\varepsilon}(\phi(x_n))$. Applying Step 1 for each $n\in \naturals$, we obtain $\phi(f_{t_n} (y))\in W^{cu}_{g,\varepsilon}(\phi(y))$ for all $y\in M$. Then, for every $k\in \mathbb{Z}$ and $n\in \naturals$ such that $k t_n \in (-\delta,\delta)$, we obtain $\phi(f_{k t_n} (y))\in W^{cu}_{g,\varepsilon}(\phi(y))$ for all $y\in M$. The set of points of the form $k t_n, k\in\mathbb{Z}, n\in \naturals$ is dense in $(-\delta,\delta)$ and the previous property passes to the closure, so we obtain 
	\[
	\phi(f_t(y)) \in W^{cu}_{g}(\phi(y))
	\]
	first for all $y\in M$ and $t\in (-\delta,\delta)$ and then also for all $t\in \reals$. We conclude that $\mathcal{W}^{cu}_f$ is mapped to $W^{cu}_g$ by $\phi$. This contradicts the assumption of the proof.
	
	For the second part, we also proceed by contradiction. Assume that $t_1(x,t)$ has the same sign in $t\in (-\delta_0,\delta _0)\setminus\{0\}$. Then, there exists times $s_1<0<s_2$ such that $s_2-s_1<\delta_0$ and $t_1(x,s_1)=t_1(x,s_2)$. But then $\phi(f_{s_1}(x))$ and $\phi(f_{s_2}(x))$ are in the same local center unstable manifold. In particular, $\phi(f_{s_2-s_1}(f_{s_1}(x)))\in W^u_{g,\varepsilon}(\phi(f_{s_1}(x)))$, which contradicts the first part.

	\textit{\textbf{Step 3.} For every $x \in M$, there exists $y \in \overline{g_{\mathbb{R}}(\phi(x))} $ such that $\mathcal{W}^s_g$ and $\mathcal{W}^u_g$ are integrable at $y$.	}
	
	Our goal is to construct surfaces that integrate the stable and the unstable. Fix $L>0$ a scale. 
	%tratar t_1 negativo	
	By the second step, we can construct sequences $0<\varepsilon_n , \varepsilon_n'<\delta_0$ and $\alpha_n\not =0$ such that $\varepsilon_n$ and $\varepsilon_n'$ tend to $0$ and 
	\[
	t_1 (x,\varepsilon_n) = \alpha_n , \quad t_1 (x,-\varepsilon_n') = -\alpha_n.
	\]
	In the following we assume that $\alpha_n$ is positive for simplicity, the negative case is done analogously.
	
	Let $h>0$ denote the topological entropy of $g_t$. Define
	\[
	T_n = \frac{1}{h} \log\left( \frac{L}{\alpha_n} \right).
	\]
	This number is chosen so that $g_{-T_n} h^+_{[-\alpha_n, \alpha_n]} \phi(x)$ is an interval of length $2L$ in the parameter of $h^+_t$. 
	
	Consider the parametrized surfaces
	\[
	\psi(t,s) := h^+_{t_1( h^-_s(x),t)} \, \phi (h^-_s (x)),
	\]
	\[
	\varphi(t,s) := g_{t_2( h^-_s(x),t)} \, h^+_{t_1( h^-_s(x),t)} \, \phi (h^-_s (x)).
	\]
	\begin{figure}[h]
		\centering
		\includegraphics[scale=1.5]{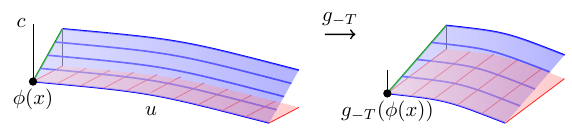}
		\caption{On the left: the red surface is the image of $\psi$ and the blue surface is the image of $\varphi$. The green segment represents the curve $t\mapsto \varphi(t,0)$. \\
			On the right: red surface, blue surface and green segment after applying $g_{-T}$ and restricting the domain. 
	} 
		\label{fig:step1}
	\end{figure}
	These surfaces are represented on the left of Figure \ref{fig:step1}.  For fixed $t\in (- \delta ,\delta )$, the curve $s\mapsto \varphi(t,s)$ runs along an unstable manifold of $g_t$. We stop at times $S_n(t),S'_n(t)>0$ such that the intervals from $\varphi(t,0)$ to $\varphi (t, S_n(t))$ and from $\varphi(t,0)$ to $\varphi (t, - S'_n(t))$ have length $L \cdot \exp(h T_n)$ in the parameter of $h^-_t$. This choices of $S_n(t)$ and $S'_n(t)$ guarantee that, for each $t\in (- \delta ,\delta )$,
	\[
	h^-_{[-L,L]} g_{-T_n} \varphi(t,0) \subset 	g_{-T_n} \varphi(\{t\}\times [-S'_n(t),S_n(t)]).
	\]
	
	Finally, define the domain
	\[
	R_n = \bigcup_{t\in[-\varepsilon'_n,\varepsilon_n]} \{t\}\times [-S'_n(t),S_n(t)].
	\]
	
	We will show that the surfaces $ g_{-T_n} \circ \psi (R_n)$ and $g_{-T_n} \circ \varphi (R_n)$ accumulate on a surface of size $2L$ foliated simultaneously by stable and unstable manifolds.  
	
	Note that $h^+_{[-L,L]}(g_{-T_n} x) \subset g_{-T_n}\psi([-\varepsilon'_n,\varepsilon_n] \times \{0\})$ and
	\[
	h^-_{[-L,L]} g_{-T_n} \varphi ([-\varepsilon'_n,\varepsilon_n] \times \{0\} ) \subset g_{-T_n} \varphi (R_n).
	\]
	
	Passing to a subsequence, we may assume $g_{-T_n} \phi(x) \to y \in M$. Then the stable segments $h^+_{[-L,L]}(g_{-T_n} \phi(x))$ converge to $h^+_{[-L,L]}(y)$, and the unstable segments $g_{-T_n} \psi (\{0\}\times [-S'_n(0),S_n(0)])$ converge to $h^-_{[-L,L]}(y)$.  
	
	Moreover, for $(s,t)\in R_n$,
	\begin{equation}\label{distance_surfaces_eng}
		d(g_{-T_n} \varphi (t,s) , g_{-T_n} \psi (t,s)) \le |g_{t_2(h_s^-(x),t)}|_\infty \le C \sup_{z\in M} |t_2(z,t)|,
	\end{equation}
	where $C$ is the supremum of the norm of the vector field generating $g_t$. Since $t_2(z,t)\to 0$ as $t\to 0$, uniformly in $z \in M$ and $t\in [ -\varepsilon'_n,\varepsilon_n]$, we deduce that the curves $g_{-T_n} \varphi ([-\varepsilon'_n,\varepsilon_n] \times \{0\} )$ also accumulate to $h^+_{[0,L]}(y)$, and the surfaces $g_{-T_n} \varphi (R_n)$ accumulate to
	\[
	S = h^-_{[-L,L]} h^+_{[-L,L]}(y).
	\]
	
	We choose the scale $L>0$ small enough such that the diameter of $S$ is less than the constant $\delta_2$ of the local product structure for the flow $g_t$. Now we show that there exists a subsurface $S'\subset S$ which is also foliated by stable segments. For every $t,s\in [-L,L]$, the local unstable manifold of $h^+_t(y)$ intersects the local center stable $W^{cs}_{loc}(h^-_s(y))$ at a unique point. Moreover, by continuity, there exists $0<\eta\le L$ such that if $s\in[-\eta,\eta]$ this intersection happens in $h^-_{[-L,L]}h^+_{t}(y)\subset S$. Define a subsurface of $S$ by 
	\[
	S' = \bigcup_{t\in[-L,L],\, s\in [-\eta,\eta]} h^-_{[-L,L]}h^+_{t}(y) \cap W^{cs}_{loc}(h^-_s(y)).
	\]
	Then $S'$ is foliated by center stable segments, obtained by intersecting $S'$ with the local center stable manifolds of $h^-_s(y)$, $s\in[-\eta,\eta]$.
	
	We now check that each of these center stable segments in $S'$ is in fact strongly stable. Fix $s\in [-\eta,\eta]$. There exists a sequence $s_n \in [-S_n(0),S_n(0)]$ such that $g_{-T_n} \varphi (0, s_n) \to h^-_s(y)$.  
	
	Observe that
	\[
	g_{-T_n}\varphi([-\varepsilon'_n,\varepsilon_n]\times \{s_n\})
	\]
	is obtained as the intersection of $g_{-T_n} \varphi([-\varepsilon'_n,\varepsilon_n]\times [-S_n(0),S_n(0)])$ with the center stable manifold of $g_{-T_n} \varphi (0, s_n)$. In the limit, this accumulates on center stable segment of $h^-_s(y)$ in $S'$. On the other hand, $g_{-T_n}\psi([-\varepsilon'_n,\varepsilon_n]\times \{s_n\})$ is a strong stable segment that converges to the same limit by (\ref{distance_surfaces_eng}). This shows that $S'$ is foliated by strong stable segments.  Hence, $S'$ is a surface transverse to $g_t$ and foliated by both stable and unstable segments.
	We have shown that there exists $y \in \overline{g_{\mathbb{R}}(\phi(x))} $ such that $\mathcal{W}^s_g$ and $\mathcal{W}^u_g$ are integrable at $y$.
	
	\textit{\textbf{Step 4.} For every $x \in M$, $\mathcal{W}^s_g$ and $\mathcal{W}^u_g$ are integrable at $x$.	
	}	

Let 
\[
U = \{ x \in M : \mathcal{W}^s_g \text{ and } \mathcal{W}^u_g \text{ are integrable at } x \}.
\]
Then $U$ is open and $g_t$-invariant. The complement $M \setminus U$ is compact and $g_t$-invariant. We have shown that $U$ is nonempty in Step 3. If $U=M$ we are done. Otherwise, consider $x \in M \setminus U$. By Step 3, there exists $y \in \overline{f_{\mathbb{R}}x} \subset M\setminus U$ such that $\mathcal{W}^s_g$ and $\mathcal{W}^s_g$ are integrable at $y$, obtaining a contradiction. We have shown that $\mathcal{W}^s_g$ and $\mathcal{W}^s_g$ are integrable, which implies that $g_t$ is a constant time suspension, contradicting that $g_t$ is topologically mixing. Hence, $\phi$ must sent the center unstable foliation $\mathcal{W}^{cu}_f$ to the center unstable foliation $\mathcal{W}^{cu}_g$.

%	Finally, since $g_t$ is transitive, this implies that $E^s$ and $E^u$ are jointly integrable at every point. %falta algo en el argumento
%	
	
%	\textbf{Argumento que sustituye el anterior:}
%	
%	\begin{enumerate}
%		\item Assume that 
%		\[
%		\phi(W^{cu}(x)) = W^{cu}(\phi(x)).
%		\]
%		Let $y \in M$. There exists $x_n \in W^{cu}(x)$ such that 
%		$\lim x_n = y$, since the central unstable manifolds is minimal.		
%		Now
%		\begin{align*}
%			\phi(W_f^{cu}(y)) &= \lim \phi\!\left(W_f^{cu}\!\left( x_n\right)\right) 
%			=  \lim W_g^{cu}(\phi(x_n)) = W_g^{cu}(\phi(y)).
%		\end{align*}
%		
%		
%		\item Show that if $\mathcal{W}_f^{cu}$ is not sent to $ \mathcal{W}_g^{cu}$, then for every $x \in M$ 
%		there exists $y \in \overline{f_{\mathbb{R}}x} $ such that $\mathcal{W}^s_g$ and $\mathcal{W}^u_g$ are integrable at $y$.
%		
%		For this we have to slightly modify the current argument, finding a point $x$ where there is transversality.
%		
%		\item Let 
%		\[
%		U = \{ x \in M : \mathcal{W}^s_g \text{ and } \mathcal{W}^u_g \text{ are integrable at } x \}.
%		\]
%		Then $U$ is open and $g_t$-invariant. We have shown that it is nonempty. The complement $M \setminus U$ is compact and $g_t$-invariant.  
%		
%		Suppose that $x \in M \setminus U$. Then there exists $y \in \overline{f_{\mathbb{R}}x} \subset M\setminus U$ such that $\mathcal{W}^s_g$ and $\mathcal{W}^s_g$ are integrable at $y$.  This is a contradiction with the fact that $M \setminus U$ is nonempty. Hence, the stable and the unstable are integrable at every point. 
%		
%
%	\end{enumerate}
	
\end{proof}

%Assuming that $g_t$ is topologically mixing, by Plante's Theorems \ref{alternative} and \ref{ji} we know that $\mathcal{W}^s_g$ and $\mathcal{W}^u_g$ are not jointly integrable, so we can rule out the second case in the dichotomy. Hence we conclude:
%
%\begin{corollary}\label{key}
%	Let $f_t, g_t$ be two $C^1$ Anosov flows on a compact $3$-manifold $M$. Assume that both $E^s_g$ and $E^u_g$ are orientable and assume that $g_t$ is topologically mixing. Let $\phi:M\to M$ be an orbit equivalence between the unstable foliation $\mathcal{W}^u_f$ of $f_t$ and the unstable foliation $\mathcal{W}^u_g$ of $g_t$. Then $\phi$ sends the center unstable foliation $\mathcal{W}^{cu}_f$ to the center unstable foliation $W^{cu}_g$.
%\end{corollary}

We can remove the orientability hypothesis on the stable and the unstable bundles.

\begin{corollary}\label{cor_key}
	Let $f_t, g_t$ be two $C^1$ Anosov flows on a compact $3$-manifold $M$. Assume that $g_t$ is topologically mixing. Let $\phi:M\to M$ be an equivalence between the unstable foliation $\mathcal{W}^u_f$ of $f_t$ and the unstable foliation $\mathcal{W}^u_g$ of $g_t$. Then $\phi$ sends the center unstable foliation $\mathcal{W}^{cu}_f$ to the center unstable foliation $\mathcal{W}^{cu}_g$.
\end{corollary}

\begin{proof}
	There exists a finite cover $\tilde M$ of $M$ where $E^s_g$ and $E^u_g$ are lifted to orientable bundles. The lift of the map $\phi$ to $\tilde M$ is in the hypothesis of Proposition \ref{key}: it is an equivalence between the unstable foliations of two topologically mixing Anosov flows. Then the lift of $\phi$ is an equivalence between the two center unstable foliations, which immediately implies that $\phi $ maps $\mathcal{W}^{cu}_f$ to $\mathcal{W}^{cu}_g$.
\end{proof}

The rest of the proof proceeds exactly along the same steps as in Abe’s article \cite{Abe95}, but we will reproduce it for completeness.

\begin{lemma}Under the hypothesis of Proposition \ref{key} or Corollary \ref{cor_key}, there exists a constant $\lambda>0$ such that for all $x\in M$ and all $t\in \reals$,
	\[
	\phi( W^u_f ( f_t (x)) ) = g_{\lambda t} ( W^u_g ( \phi (x))).
	\]
\end{lemma}
\begin{proof}
	Since $\phi$ sends the center unstable foliation of $f_t$ to the center unstable foliation of $g_t$, we know that $\phi(f_t(x))$ belongs to the center unstable manifold of $g_t$ at $\phi(x)$ for all $x\in M$ and $t\in \reals$. In particular, if $t\in (-\delta, \delta)$, (\ref{basic_formula}) holds and the time $t_1(x,t)=0$ for all $x\in M$. So for all $(x,t)\in M \times (-\delta, \delta)$, we have
	\[
	\phi(f_t(x)) = h^-_{t_3(x,t)} \, g_{t_2(x,t)}  \, \phi(x).
	\]
	Now it is immediate that for, fixed $t\in (-\delta,\delta)$, the function $x\mapsto t_2(x,t)$ is constant along the unstable manifolds ${W}^u_f(x)$ of $f_t$. Recall that the fact that $g_t$ is topologically mixing implies that the foliation $\mathcal{W}_g^u$ is minimal, and so is the foliation $\mathcal{W}^u_f$ because they are equivalent. We deduce that $t_2(x,t)$ is a constant function of $x$, hence there exists a continuous map $\bar t_2 : (-\delta, \delta) \to \reals $ such that $t_2(x,t) = \bar t_2(t)$. 
	
	Given $t,t'\in (-\delta,\delta)$ such that $t+t'\in (-\delta, \delta)$, for any $x\in M$ we have
	\[
	\phi(f_{t'}(x))\in W^u ( g_{\bar t_2 ( t' )} \phi(x)) , \quad \phi(f_{t+t'}(x)) \in W^u ( g_{\bar t_2 ( t+t' )} \phi(x)).
	\]
	But then also
	\[
	\phi(f_{t}(f_{t'}(x)))\in W^u ( g_{\bar t_2 ( t )} \phi(f_{t'}(x) )) = W^u ( g_{\bar t_2 ( t ) + \bar t_2 (t')} \phi(x )).
	\]
	We deduce that $\bar t_2 (t+t') =\bar t_2 (t) +\bar t_2 (t')$ (see Figure \ref{fig:step2}). 
\begin{figure}[h]
	\centering
	\includegraphics{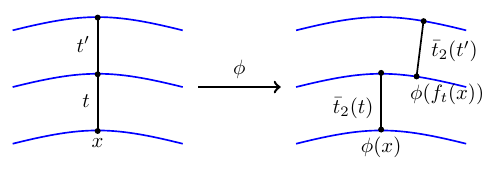}
	\caption{Additivity of $\bar t_2$.}
	\label{fig:step2}
\end{figure}
	Since $\bar t_2$ is continuous, in fact it must be linear: there exists $\lambda\in \reals $ such that $\bar t_2 (t) = \lambda t$ for all $t\in (-\delta,\delta)$. This shows that, for all $x\in M$ and all $t\in (-\delta,\delta)$,
	\begin{equation}\label{eq3}
		W^u_g (\phi(f_t(x))) = g_{\lambda t }  (W^u_g (\phi(x))).
	\end{equation}
	Given $t\in \reals$, considering a number $n\in \mathbb{N}$ such that $t/n\in (-\delta,\delta)$ and applying the previous line to the points $f_{kt/n}(x), \, 0\le k\le n-1$ and the time $t/n$, allows to extend equality (\ref{eq3}) for all $t\in \reals$. 
	
	Finally, since $f_t$ expands the unstable foliation $\mathcal{W}^u_f$, $g_{\lambda t}$ must expand the foliation $\mathcal{W}^u_g$, so $\lambda$ is positive. 
\end{proof}

We finally construct a topological conjugacy between the flows $f_t$ and $g_{\lambda t}$.
For $T\ge0 $, consider the maps $\phi_T: M\to M$ defined by 
\[
\phi_T = g_{-\lambda T} \circ \phi \circ  f_T.
\]
It is clear that each map $\phi_T$ is a homeomorphism that sends $W^u_f(x)$ to $W^u_g(\phi(x))$ for every $x\in M$.

\begin{lemma}
	The maps $\phi_T$ converge uniformly when $T\to \infty$ to a homeomorphism $\psi:M\to M$ that satisfies  $\psi \circ f_t = g_{\lambda t} \circ \psi$ for all $t\in \reals$.
\end{lemma}
\begin{proof}
	We observe again that $\phi_T(x)$ belongs to $W^u_g(\phi(x))$ for every $x\in M$ and $T\ge0$, so there exists $\tau (x, T)\in \reals $ such that $\phi_T (x) = h^-_{\tau(x,T) }(\phi(x) )$. Moreover the map $(x,T) \to \tau (x,T)$ is continuous. Consider 
	\[
	\tau_0 = \sup \{ \tau(x, T) : x\in M,\, T\in [0,1]\}.
	\]
	We will show that $\tau(x,T)$ are Cauchy when $T\to \infty$ uniformly in $x\in M$. A short computation shows that $\tau$ satisfies the property
	\[
	\tau(x,T_1+T_2) = \tau (x,T_1) + e^ {-\lambda T_1 } \tau (f_{T_1} (x), T_2 ) .
	\]
	If $T\ge T'\ge 0$ and $T-T'\le 1$, then
	\[
	|\tau(x,T) - \tau(x,T')| = e^ {-\lambda T' } |\tau (f_{T'} (x), T-T' ) |\le e^ {-\lambda T' } \tau_0.
	\]
	Therefore, for any $T\ge T'\ge 0$, 
	\[
	|\tau(x,T) - \tau(x,T')| \le  \sum_{k= \lfloor T' \rfloor }^{\lfloor T \rfloor} e^ {-\lambda k } \tau_0
	\le \frac{e^ {-\lambda \lfloor T' \rfloor } \tau_0}{1- e^ {-\lambda }} ,
	\]
	which shows that $\tau$ is uniformly Cauchy. Let $\tau (x) $ be the limit of $\tau(x,T)$ when $T\to +\infty$. It is clear that $\phi_T(x)$ converges to $\psi (x) = h^-_{\tau(x)} (\phi(x))$ uniformly in $x\in M$. 
	
	Given $t\in \reals$, if $T\ge0$ is big enough such that $T+t$ is positive, we have
	\[
	\phi_T \circ f_t = g_{-\lambda T} \circ \phi \circ f_{T+t} = g_{\lambda t} \circ \phi_{T+t}.
	\]
	Taking the limit when $T\to +\infty$ we obtain $\psi \circ f_t = g_{\lambda t} \circ \psi$.
	
	We finally prove that $\psi$ is a homeomorphism. Notice that $\phi_T$ is an homeomorphism for all $T\ge0$. This already implies that $\psi$ is surjective. For injectivity, assume that there are points $x,y\in M$ such that $\psi(x)=\psi(y)$. Applying the conjugacy property we have 
	\[
	\psi(f_t(x))=g_{\lambda t} \psi(x)= g_{\lambda t} \psi(y) = \psi (f_t(y))
	\] for all $t\in \reals$. From the definition of $\psi$ we obtain
	\[
	\phi(f_t(y)) = h^u_{\tau(f_t(x))-\tau(f_t(y))} \phi(f_t(x)) 
	\]
	The quantity $\tau(f_t(x))-\tau(f_t(y))$ is bounded for $t\in \reals$. Applying the inverse of $\phi$, we deduce that $x$ and $y$ are in the same unstable manifold of $\mathcal{W}^u_f$ and that their orbits $f_t(x)$ and $f_t(y)$ remain at bounded unstable distance for all $t\in \reals$. This can only happen if $x=y$ because $\mathcal{W}^u_f$ is expanding.  
	Since $\psi $ is continuous and bijective and $M$ a compact manifold, $\psi $ is automatically a homeomorphism.

\end{proof}

\subsection{Constant time suspensions of Anosov diffeomorphisms}

%This case is probably not surprising for specialists. 

As in the previous section, let $f_t, g_t$ be two $C^1$ transitive Anosov flows on a compact $3$-manifold $M$. Assume that there is a homeomorphism $\phi: M\to M$ which sends the unstable foliation $\mathcal{W}^u_f$ of $f_t$ to the unstable foliation $\mathcal{W}^u_g$ of $g_t$.

Assume that $f_t$ is a constant time suspension of an Anosov diffeomorphism. 
%Since $\mathcal{W}_g^u$ is not minimal, $\mathcal{W}_f^u$ is neither and $f_t$ is also a constant time suspension of an Anosov differomorphism. 
Plante shows that for any point $x\in M$, $S=\overline{W^u_f(x)}$ is a $C^1$ section of the flow $f_t$. In fact there exists a first return time $T>0$ such that $f_T S= S$ and $F=f_T$ restricted to $S$ is an Anosov diffeomorphism. Moreover $M$ is a bundle over the circle $S^1$ with monodromy map $F:S\to S$. 	

The map $\phi$ sends $S$ to a corresponding section $S' =\overline{W^u_g(\phi(x))}$ of $g_t$ and gives $M$ another structure of bundle over the circle. The monodromy map is conjugated to $F$ by $\phi$ but it also must be equal to $g_{T'}$ for some $T'\not = 0$. Since $\phi $ maps the unstable foliation of $f_t$ to the unstable foliation of $g_t$, then $T'$ must be positive. Hence $f_T$ is topologically conjugated to $g_{T'}$ and this property passes to the supensions flows $f_t$ and $g_{\lambda t}$ with $\lambda $ given by $T'/T>0$.

	\bibliographystyle{alpha}
	\bibliography{general_library}
\end{document}